\documentclass[12pt,a4paper]{article}

\usepackage[biblatex]{anziamjedraft}
\bibliography{eg}

\newcommand{\E}{\mathbb{E}}
\newcommand{\R}{\mathbb{R}}

\usepackage{xcolor}
\usepackage{graphicx}
\usepackage{listings}
\usepackage[ruled, vlined, boxed]{algorithm2e}
\usepackage{amsfonts}

\title{Computing expected moments of the R\'enyi 
 parking problem on the circle}

\author{M.~Hegland}
\address{Mathematical Sciences Institute, The Australian National University, 
Australian Capital Territory~2601, \textsc{Australia}.}
\mailto{markus.hegland@anu.edu.au}
\myorcid{0000-0002-5136-2883}

\author{C.J.~Burden}
\address{Mathematical Sciences Institute, The Australian National University, 
Australian Capital Territory~2601, \textsc{Australia}.}
\mailto{conrad.burden@anu.edu.au}
\myorcid{0000-0003-0015-319X}

\author{Z.~Stachurski}
\address{School of Engineering, The Australian National University, 
Australian Capital Territory~2601, \textsc{Australia}.}

\begin{document}

\maketitle

\begin{abstract}
    A highly accurate and efficient method  
    to compute the expected values of the count, sum, and squared norm of the sum of
    the centre vectors of a random maximal sized collection of non-overlapping unit diameter disks touching a fixed unit-diameter disk is presented. This extends earlier work on 
    R\'enyi's parking problem [Magyar Tud. Akad. Mat. Kutat\'{o} Int. K\"{o}zl. 3 (1-2), 1958, pp. 109-127]. Underlying the method is a splitting of the the problem conditional 
    on the value of the first disk. This splitting is proven and then used to derive integral equations for the expectations. These equations take a lower block triangular form. 
    They are solved using substitution and approximation of the integrals to very high accuracy using a polynomial approximation within the blocks.  \end{abstract}

\tableofcontents

%
%
\section{Introduction}

Consider the two-dimensional process of unit-diameter disks attaching randomly onto the circumference of a fixed central unit-diameter disk, as shown in 
Figure~\ref{disksAndCars}(a).  The accretion is assumed to occur sequentially in such a way that the location of each additional disk is randomly and uniformly 
distributed over the available accessible part of the central disk.  The process stops when no further attachment points are available.  Herein we analyse the 
distributional properties of the vector sum of the location of the centres of the attached disks.  
    
The problem is related to a well-known one-dimensional packing problem known as R\'{e}nyi's car parking problem~\citep{ClaSN16}.  
\citet{Ren58}  (see \citep[pp 203--218]{SelectedTrans} for an English translation) 
analysed a stochastic process arising when sequentially packing unit intervals $[Y_i,Y_i+1)$ representing parked cars into an interval $[a,b]$ 
(where $a,b\in\R$ and $Y_i\in [a,b-1]$ is a random number), such that the unit intervals do not overlap. The location of each additional car is randomly and uniformly distributed over the available subset of $[a, b]$,
and the packing is completed when all the remaining gaps are less than one.  
R\'enyi showed~\cite{Ren58} that the expected value of the number of parked cars satisfies 
an integral equation resulting from a splitting property of the number of cars in sub-intervals on either side of the first parked car. He further established an exact 
value (approximately equal to $0.75$) for the limit of the expected packing density as the parking interval $b - a \to \infty$.  
After the first external disk has attached, 
our approach involves of mapping the available part of the circumference of the central disk 
onto the interval $[1, 6]$, as shown in Figure~\ref{disksAndCars} (b) for the interval case which is appropriately mapped to the circular case in Figure~\ref{disksAndCars} (a), and developing analogues of R\'enyi's integral equation for features relevant to the disk accretion 
problem.  

The expected maximal number of parked cars is computed explicitly for small $x = b - a$.  However the formulas get very complicated even for 
moderate $x$, take a long time to compute and often evaluate poorly due to rounding errors. Thus numerical approaches were established for computing the 
expectation~\cite{LalG74, Maretal89}.  These numerical approaches give highly accurate results. The approach taken here is similar in that it also uses a 
highly accurate piecewise polynomial approximation. In addition to computing the expected number of added disks, we compute the mean squared 
shift of the vector sum of centres of the attached disks. This requires the solution to high numerical accuracy of a 
system of integral equations.  We review the underlying splitting property in Section~\ref{sec:2} for various features relevant to the disk accretion problem and derive 
a system of integral equations for their expected values in Section~\ref{sec:3}. The method is then implemented using the new framework developed by \citet{OlvT14}. 
The algorithm is summarised in Section~\ref{sec:4} and final results are presented in Section~\ref{sec:5}.

\begin{figure}
        \centering
         \includegraphics[width=0.8\textwidth]{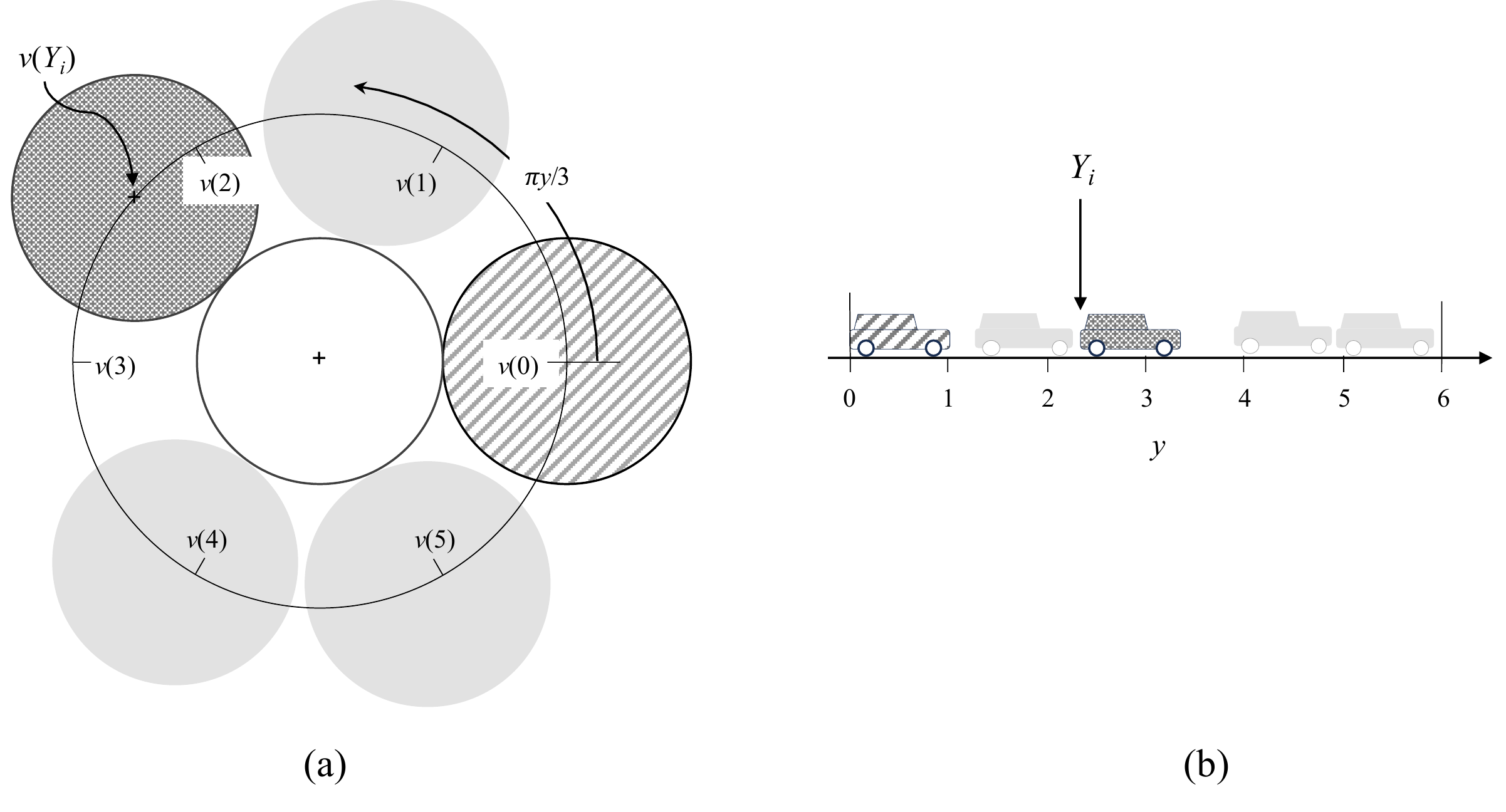}
        \caption{(a) The accretion of disks to a central disk, and (b) the equivalent R\'{e}nyi parking problem, in which the centre of each disk is mapped to the left hand 
        end of a unit-length car.  The vector-valued function $v$ is defined by Equation~(\ref{eqn:v}).  Without loss of generality, the first attached disk is assumed to be 
        centred at $v(0) = (1, 0)^{\rm T}$.  Equivalently, the first car is assumed to be 
        located at position 0, and hence the R\'{e}nyi process pertains to the interval $[1, 6]$.
}
        \label{disksAndCars}
\end{figure}

%
%
\section{Mathematical modelling\label{sec:2}}

%
%
\subsection{R\'{e}nyi Point Process for $[a,b]$}

Let $K(a,b)$ be the (random) number of unit-length cars parked in an interval $[a, b]$ when packing is completed. R\'enyi's stochastic process defines a 
random set 
\[
\alpha(a,b)=\{Y_1,\ldots,Y_{K(a,b)}\} \subset [a,b-1],
\] 
which is a translational invariant point process, as shown below where the random variables $Y_i$ are formally introduced.

\begin{lemma}[Translational Invariance]\label{lemma:translation}
    \[\alpha(a+t,b+t) = t + \alpha(a,b), \; \text{for all $t\in\R$}.\]
\end{lemma}
\begin{proof}[Proof idea]
    The stochastic process for the interval $[a+t,b+t]$ generates a sequence of 
    random variables $Y_1+t,\ldots,Y_K+t$ where the $Y_i$ are the random variables generated by the process for the interval $[a,b]$.
\end{proof}
It follows that $K(a,a+x)=K(0,x)$ for all real $x\geq 0$. 

A consequence of the one-dimensionality of the R\'enyi point process is  
a splitting property of the random set conditional
on the first point $Y_1$. This conditional set is denoted by
$\alpha(a,b\mid Y_1=y)$. The splitting property gives rise to
the integral equations derived in the next section.
\begin{lemma}[Splitting Property] \label{lemma:splitting}
  \[\alpha(a,b\mid Y_1=y) = \{y\} \cup \alpha(a,y) \cup \alpha(y+1,b).\]
\end{lemma}
\begin{proof}
  When the random variables $Y_2,Y_3,\ldots$ are generated, each element
  is either in $[a,y-1]$ or in $[y+1,b-1]$. Now $Y_i, Y_j$ are independent
  if $Y_i\in[y+1,b-1]$ and $Y_j\in[a,y-1]$ as they cannot overlap. Consequently, the points in the two subsets generated for the conditional process produce two independent random 
  sets $\alpha(a,y)$ and $\alpha(y+1,b)$.
\end{proof}

The R\'enyi point process gives rise to a \emph{random counting measure} with parameters 
$a,b$ defined by
\[N(a,b) = \sum_{y\in\alpha(a,b)} \delta_y\]
where $\delta_y$ is the atomic measure for which
\[
\delta_y(M)=\begin{cases}
			1, & \text{if $y\in M$}\\
            0, & \text{otherwise}
		 \end{cases}
\]
for any set $M\subset \R$.
 It follows that the
support of this measure is $\alpha(a,b)$. This measure is defined for any set $M \subset \R$
by
\begin{equation}
    N(a,b)(M) = \lvert \alpha(a,b)\cap M\rvert.
\end{equation}
In particular, one has $K(a,b)=N(a,b)([a,b])$. Note that $K$ is translational invariant, 
that is $K(a+t,b+t)=K(a,b)$. 

The conditional counting measure $N(a,b \mid Y_1=y)$ satisfies a splitting property
\begin{equation}\label{eqn:Nsplitting}
  N(a,b\mid Y_1=y) = \delta_y + N(a,y) + N(y+1,b).
\end{equation}
This is a consequence of Lemma~\ref{lemma:splitting}.

%
%
\subsection{Randomly packing disks touching a central disk}

We map the interval $[0,6)$ bijectively onto the unit circle $S^2$ by
\begin{equation}	\label{eqn:v} 
v(y)= Q(y) \begin{bmatrix}1\\0\end{bmatrix}, \quad y\in[0,6).
\end{equation}
where
\begin{equation}
Q(y) = \begin{bmatrix}\cos(\pi y/3)& -\sin(\pi y/3)\\ \sin(\pi y/3) 
	& \cos(\pi y/3)\end{bmatrix},\label{eqn:Q}
\end{equation}
as shown in Figure~\ref{disksAndCars}.  
Packing subintervals of $[1,6)$ then corresponds to packing unit diameter disks 
touching a central unit diameter disk, conditional on an initial disk being attached with its centre at $v(0) = (1, 0)^{\rm T}$. The centres of the touching disks are on the
unit-radius circle $S^2$. One can thus apply R\'enyi's point processes to the addition of disks subsequent to the first added disk. 

We now introduce real and vector valued features of the point process which are polynomial 
functions of the random variables
$v(Y_1),\ldots,v(Y_K)$ invariant under permutations of the $Y_i$. In particular we consider
features invariant under any translation of the parameters $a$ and $b$. Clearly, $K(a,b)$ is an example
as $K(a+t,b+t)=K(a,b)$ and $K(a,b)$ is a constant function of the $Y_i$. 
We introduce a 2-component vector-valued polynomial
\begin{equation} F(a,b) = \sum_{i=1}^{K(a,b)} v(Y_i) \label{Fab}.\end{equation}
The value $((1, 0)^{\rm T} + F(1, 6))/(1 + K(1, 6))$ is the centre of mass of the circular polygon with vertices $v(Y_i)$. 
The splitting property of $F(a,b)$ follows from the definition and is
\begin{equation} 
F(a,b \mid Y_1=y) = F(a,y) + F(y+1,b) + v(y).\label{eqn:Fsplitting}
\end{equation}
But $F$ is not invariant under translations since 
\begin{equation}F(a+t,b+t) = Q(t) F(a,b).\label{eqn:Ftranslate}\end{equation}
We introduce the feature
\begin{equation}	\label{LsqDef}
L^2(a,b) = \lVert F(a,b)\rVert^2
\end{equation}
which is used in Section~\ref{sec:Results} to determine the expected shift in the centre of the polygon with vertices $v(Y_i)$.  
We introduce another second degree feature 
\begin{equation} E_2(a,b) = \sum_{Y_i<Y_j} v(Y_i)^T v(Y_j) = \sum_{Y_i<Y_j} \cos(\pi (Y_i-Y_j)/3)
\label{eqn:E2}\end{equation}
which is better suited to the computations performed in Section~\ref{sec:3}. We show below how this feature relates to $L^2(a,b)$. This feature is also invariant under translations of the $Y_i$ as it only depends on the differences $Y_i-Y_j$ and the translated difference is
$(Y_i+t)-(Y_j+t) = Y_i-Y_j$.
This feature admits the following splitting property:
\begin{proposition}
    \begin{equation}\label{eqn:E2splitting}
      \begin{split}
       E_2(a,b \mid Y_1=y) = E_2(a,y) + E_2(y+1,b) + F(a,y)^T F(y+1,b) \\ 
         + v(y)^T (F(a,y)+F(y+1,b))
      \end{split}
    \end{equation}
\end{proposition}
\begin{proof}
   The sum in Equation~(\ref{eqn:E2}) decomposes into three sums over $i, j \neq 1$, one sum over $i$ where $j=1$ and one sum over $j$ where $i=1$ as follows
   \[\sum_{Y_i<Y_j} = \sum_{Y_i<Y_j<y} + \sum_{y<Y_i<Y_j} + \sum_{Y_i<y<Y_j} + \sum_{Y_1=y<Y_j} + \sum_{Y_i<y=Y_1}.\]
The summand in each term is $\cos(\pi(Y_i-Y_j)/3)$.
\end{proof}
We use the splitting property to derive integral equations in the next section for the expectation of $F(a,b)$ and of $E_2(a,b)$. 

We now have three features $K, L^2$ and $E_2$ connected by the following Lemma. 
\begin{proposition}	\label{prop4}
 \[L^2(a,b) = K(a,b) + 2 E_2(a,b).\]  
\end{proposition}
\begin{proof}
    One has
    \begin{align*}
     L^2(a,b) & = \sum_{i=1}^{K(a, b)} \sum_{j=1}^{K(a, b)} v(Y_i)^T v(Y_j) \\
              & = \sum_{i=1}^{K(a, b)} \lVert v(Y_i)\rVert^2 +
                  2 \sum_{Y_i<Y_j} v(Y_i)^T v(Y_j).
    \end{align*}
    The result follows as $v(y)$ lies on the unit circle.
\end{proof}

%
%
\section{Integral equations\label{sec:3}}

%
%
\subsection{Equation for $\E K(0,x)$}

The random variables $K(a,b)$ denoting the counts of parked vehicles (or disks) in R\'enyi's
model on an interval $[a,b]$ are translational invariant, that is $K(a,b) = K(a+t,b+t)$ for any $t\in\R$. 
An application of these properties provides the following lemma.
\begin{proposition}
  Let $u_1(x)=\E K(0,x)$ be the expected number of vehicles in the interval $[0,x]$. Then
      \begin{equation}	\label{u1Eqn}
      u_1(x) = \frac{2}{x-1}\int_0^{x-1} u_1(y)\, dy + 1.
      \end{equation}
\end{proposition}

\begin{proof}
    Let $Y_1$ be the (random) position of the first vehicle. For any $a\leq y < b$, 
    Equation~(\ref{eqn:Nsplitting}) implies that the   
    counting variable $K(a,b)$ conditional on $Y_1=y$ satisfies a  \emph{splitting property} 
    \[K(a,b \mid Y_1=y) = K(a,y) + K(y+1,b) + 1.\]
    As $Y_1$ is uniformly distributed over $[a,b-1]$ one has 
    \[ \E K(a,b) = \frac{1}{b-1-a}\int_a^{b-1} \E K(a,b \mid Y_1=y)\, dy. \]
    As the counting variable is translational invariant, that is $K(a,b)=K(0,b-a)$ one solves the proposition by substituting the 
    conditional count on the right-hand side using the splitting property and substituting $a=0$ and $b=x$.
\end{proof}

The expected count for a general interval is
\[\E K(a,b) = u_1(b-a).\]

%
%
\subsection{Equation for $\E F(0,x)$}
         
\begin{proposition}
    Let the (vector valued) function  $u_2(x)=\E F(0,x)$ where $F(0,x)$ is defined in Equation~(\ref{Fab}). Then
    \begin{equation}	\label{u2Eqn}
    u_2(x) = \frac{1}{x-1} \int_0^{x-1} \left(u_2(y) + Q(x-y)u_2(y) + v(y)\right)\, dy 
    \end{equation}
    where $Q$ is defined in Equation~(\ref{eqn:Q}).
\end{proposition}

\begin{proof}
    The splitting condition~(\ref{eqn:Fsplitting}) is
    \[F(0,x\mid Y_1=y)=F(0,y)+F(y+1,x)+v(y).\]
    Using the translation formula Equation~(\ref{eqn:Ftranslate}) for $F$ and taking the expectation gives
    \[u_2(x\mid Y_1=y)=u_2(y) + Q(y+1)u_2(x-1-y)+v(y).\]
    Integrating over $y$ and reversing the order of the integration for the second
    term gives the result.
\end{proof}

%
%
\subsection{Equation for $\E E_2(0,x)$}

\begin{proposition}
    Let $u_3(x)=\E E_2(0,x)$. Then
    \begin{equation}	\label{u3Eqn}
    u_3(x)=\frac{2}{x-1}\int_0^{x-1} u_3(y)\, dy + g_3(x)
    \end{equation}
where
  \begin{equation*}\begin{split}
    g_3(x) = \frac{1}{x-1}\int_0^{x-1}\left(u_2(y)^TQ(y+1)u_2(x-1-y) + v(y)^T u_2(y)\right. \\ \left. +\,v(-1)^T u_2(y)\right)\, dy
  \end{split}\end{equation*}
\end{proposition}

\begin{proof}
  The splitting of $E_2$, Equation~(\ref{eqn:E2splitting}) gives
  \begin{equation*}\begin{split}
      E_2(0,x\mid Y_1=y)=E_2(0,y)+E_2(y+1,x)+F(0,y)^TF(y+1,x) \\
          +\, v(y)^T\left(F(0,y)+F(y+1,x)\right).
  \end{split}\end{equation*}
  Using the independence of  $F(0,y)$ and $F(y+1,x)$ one gets
  \begin{equation*}\begin{split}
      u_3(x\mid Y_1=y)= u_3(y)+u_3(x-1-y)+u_2(y)^T Q(y+1)u_2(x-1-y) \\
          +\, v(y)^T\left(u_2(y)+Q(y+1)u_2(x-1-y)\right).
  \end{split}\end{equation*}  
  Noting that $v(y)^T Q(y+1) = v(-1)^T$ and integrating over $y$ with a change of integration variable in the last term gives the result
\end{proof}

%
%
\section{Numerical solution\label{sec:4}}

Equations~(\ref{u1Eqn}), (\ref{u2Eqn}) and (\ref{u3Eqn}), for the expectations $u_1(x) = \E K(0,x)$, $u_2(x) = \E F(0,x)$ and $u_3(x) = \E E_2(0,x)$ 
are all integral equations of the form
\begin{equation}	\label{eqn:ie}
   u = \mathcal{L}u + g
\end{equation}
for some linear integral operator $\mathcal{L}$ and function $g$ specified in the following
sub-sections.
In the R\'{e}nyi process, a space of length less than 1 must be unoccupied, so $u(x)=0$ for $x\in[0,1)$.  Furthermore, 
\[
\mathcal{L}u(x) =  \frac{1}{x-1} \int_0^{x-1} M(x,y) u(y)\, dy, \quad \text{for $x>1$},
\]
where $M(x,y) = I + Q(x-y)$ 
for the vector case $u_2(x)$,
and $M(x,y)=2$ for the scalar cases $u_1(x)$ and $u_3(x)$.
We are interested in computing $u(x)$ for $x\in[0,5]$.
The solution of Equation~(\ref{eqn:ie}) follows trivially for $x\in[0,2)$: 
\[
u(x) = \begin{cases}
	0, & x\in[0,1); \\
	g(x), & x\in[1,2).
	\end{cases}
\]
Now if one interprets the function $u(x)$ as a block vector where each block corresponds to a function on an interval $[k-1,k]$ then the integral equation~(\ref{eqn:ie}) is of block lower triangular structure and can be solved by repeated substitution. Specifically, one computes 
$u(x)$ for $x\in[k,k+1]$ by
\[u(x) = \frac{1}{x-1} \int_0^{x-1} M(x,y) u(y) + g(x), \quad x\in[k,k+1)\]
which uses $u(x)$ for $x\in[k-1,k]$. It then follows directly that $u$ is $C^\infty$ on each interval $[j-1,j]$ for $j=1,2,\ldots$ and is continuous on $[1,\infty)$.

The approach was implemented in the Julia language~\cite{Bezetal17} using the ApproxFun package~\cite{OlvT14,Olv22}. The code is available on request from the first author.  
Plots of the functions $u_1$, $u_2$ and $u_3$ are shown in Figure~\ref{u_plots}.  Computed values agree to machine accuracy with analytic calculations of 
$u_1(x)$ for $0 \le x \le 4$ \citep{Wei69}, and our own analytic calculations of $u_1(5)$ and $u_2(x)$ and $u_3(x)$ for $0 \le x \le 3$.  

\begin{figure}
        \centering
        \includegraphics[width=0.75\textwidth]{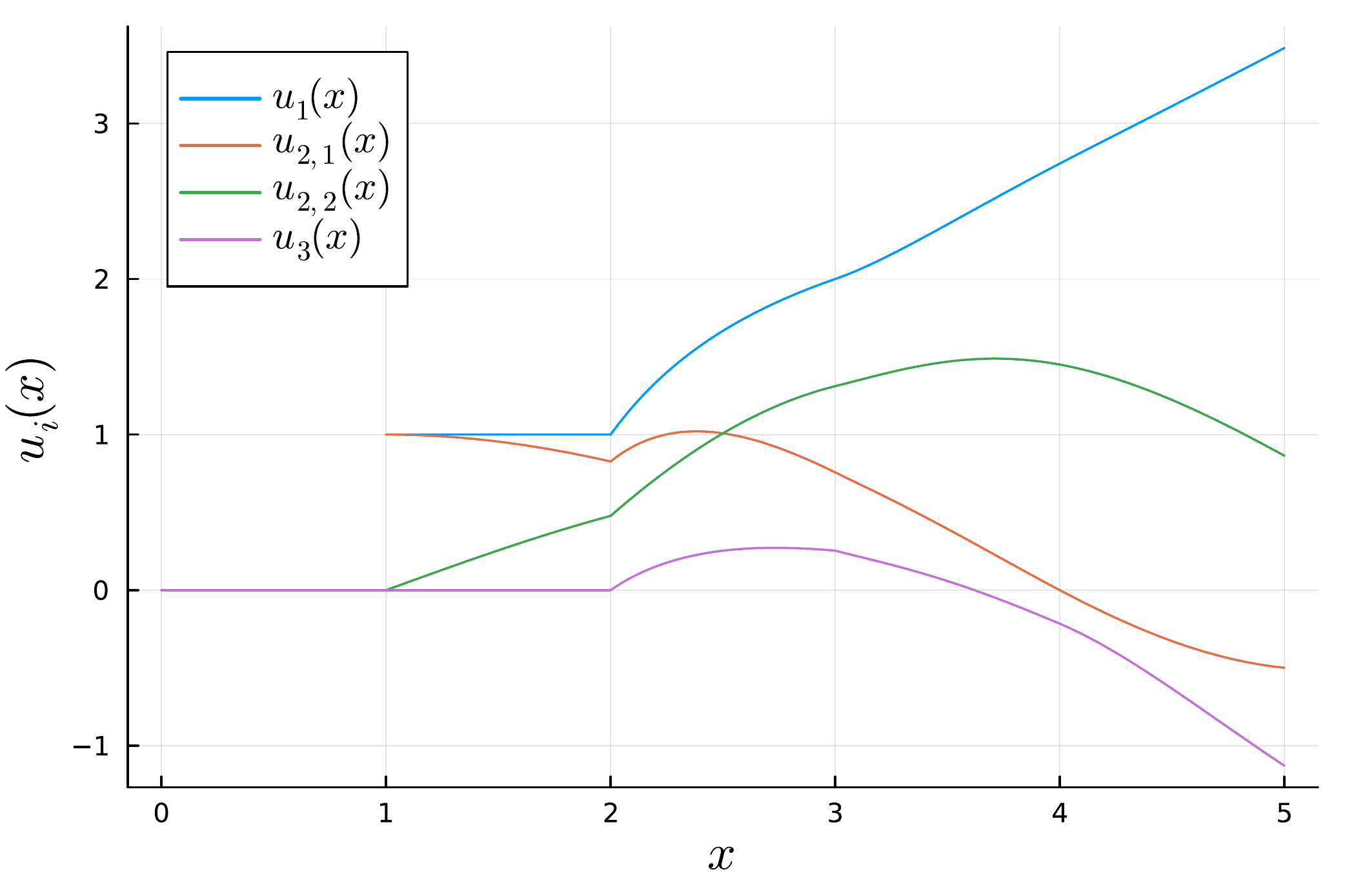}
        \caption{Numerical solutions to  $u_1(x)$, both components of $u_2(x)$, and $u_3(x)$.  }
        \label{u_plots}
\end{figure}

%
%

\section{Results\label{sec:5}}	\label{sec:Results}

It remains to compute the expected values of features pertaining to the full set of attached disks.  The expected total number of disks is 
\[
\E[1 + K(1, 6)] = 1 + u_1(5) = 4.48508592498075.  
\]
 The expected vector sum of centres of the attached discs conditional on the first disk being located at $(1, 0)^{\rm T}$ is 
\begin{eqnarray*}
\lefteqn{ \E\left[\begin{pmatrix}1\\0\end{pmatrix} + F(1, 6) \right]} \\
	& = & \begin{pmatrix}1\\0\end{pmatrix} + Q(1) \E[F(0, 5)] \\
	& = & \begin{pmatrix}1\\0\end{pmatrix} + \begin{pmatrix} \cos(\pi/3) & -\sin(\pi/3) \\ \sin(\pi/3) & \cos(\pi/3) \end{pmatrix} u_2(5)  \\
	& = & \begin{pmatrix} \text{0.00226785060421} \\ 0 \end{pmatrix}.  	
\end{eqnarray*}
By symmetry, the second component must be zero, and this is confirmed to machine accuracy.  
Finally, define the mean square shift of the sum of centres of the attached disks as 
\begin{eqnarray*}
I &  := & \left\lVert \begin{pmatrix}1\\0\end{pmatrix} + F(1, 6) \right\rVert^2 \\
	& = & 1 + 2 (1, 0) Q(1)F(0, 5) + \lVert F(0, 5) \rVert^2 \\
	& = & 1 + 2 \left(\cos(\pi/3), -\sin(\pi/3) \right) F(0, 5) + L^2(0, 5).  
\end{eqnarray*}
Its expected value is, using Proposition~\ref{prop4}, 
\[ 
\E[I] = 1 + 2 \left(\cos(\pi/3), -\sin(\pi/3) \right) u_2(5) + u_1(5) + 2u_3(5) = 0.2325047203936.  
\]

 \paragraph{Acknowledgements}

An initial version of the code was developed and tested using the cloud service for VSCode on juliahub.com.

\ifx\printbibliography\undefined
    \bibliographystyle{plain}
    \bibliography{eg}
\else\printbibliography\fi

\end{document}